\documentclass{amsart}
\usepackage{amssymb}
\usepackage{amsfonts}
\usepackage{amssymb}
\usepackage{amsmath}
\usepackage{amsthm}
\usepackage{enumerate}
\usepackage{tabularx}
\usepackage{centernot}
\usepackage{mathtools}
\usepackage{stmaryrd}
\usepackage{amsthm,amssymb}
\usepackage{etoolbox}
\usepackage[a]{esvect}

\def\sp{\mathrm{Sp}}

\def\U{\mathcal{U}}
\def\F{\mathcal{F}}
\def\Q{\mathcal{Q}}
\def\ptl{\mathrm{PTL}}
\usepackage{tikz}
\usetikzlibrary{decorations.pathmorphing,shapes}

\newcounter{sarrow}

\renewcommand{\leq}{\leqslant}
\renewcommand{\geq}{\geqslant}

\makeatletter
\def\subsection{\@startsection{subsection}{3}%
  \z@{.5\linespacing\@plus.7\linespacing}{.3\linespacing}%
  {\bfseries\centering}}
\makeatother

\makeatletter
\def\subsubsection{\@startsection{subsubsection}{3}%
  \z@{.5\linespacing\@plus.7\linespacing}{.3\linespacing}%
  {\centering}}
\makeatother

\makeatletter
\def\myfnt{\ifx\protect\@typeset@protect\expandafter\footnote\else\expandafter\@gobble\fi}
\makeatother

\theoremstyle{definition}

\newtheorem{theorem}{Theorem}[section]
\newtheorem{definition}[theorem]{Definition}
\newtheorem{lemma}[theorem]{Lemma}

\newtheorem{corollary}[theorem]{Corollary}

\newtheorem{remark}[theorem]{Remark}

\newcounter{claimcounter}
\numberwithin{claimcounter}{theorem}
\newenvironment{claim}{\stepcounter{claimcounter}{\noindent {\bf Claim \theclaimcounter.}}}{}
\newenvironment{claimproof}[1]{\noindent{{\em Proof.}}\space#1}{\hfill $\rule{0.35em}{0.35em}$}

\newcommand{\pureindep}[1][]{%
  \mathrel{
    \mathop{
      \vcenter{
        \hbox{\oalign{\noalign{\kern-.3ex}\hfil$\vert$\hfil\cr
              \noalign{\kern-.7ex}
              $\smile$\cr\noalign{\kern-.3ex}}}
      }
    }\displaylimits_{#1}
  }
}

\newcommand{\indep}[2]{%
  \mathrel{
    \mathop{
      \vcenter{
        \hbox{%
\oalign{
\noalign{\kern-.3ex}\hfil$\vert$\hfil\cr
              \noalign{\kern-.7ex}
              $\smile$\cr\noalign{\kern-.3ex}
}
}
      }
}^{\!\!\!\!\!#2}_{\!\!\hspace{-0.1em}#1}
  }
}

\newcommand{\displayindep}[2]{%
  \mathrel{
    \mathop{
      \vcenter{
        \hbox{%
\oalign{
\noalign{\kern-.3ex}\hfil$\vert$\hfil\cr
              \noalign{\kern-.7ex}
              $\smile$\cr\noalign{\kern-.3ex}
}
}
      }
}^{\!\!\hspace{-0.1em}#2}_{\!\!\hspace{-0.1em}#1}
  }
}

\def\presuper#1#2%
  {\mathop{}%
   \mathopen{\vphantom{#2}}^{#1}%
   \kern-\scriptspace%
   #2}

\begin{document}

\title{Quantum Team Logic and Bell's Inequalities}
\thanks{The research of the second author was supported by the Finnish Academy of Science and Letters (Vilho, Yrj\"o and Kalle V\"ais\"al\"a foundation) and a grant TM-13-8847 of CIMO. The research of the third author was partially supported by
grant 251557 of the Academy of Finland.}

\author{Tapani Hyttinen}
\address{Department of Mathematics and Statistics,  University of Helsinki, Finland}

\author{Gianluca Paolini}
\address{Department of Mathematics and Statistics,  University of Helsinki, Finland}

\author{Jouko V\"a\"an\"anen}
\address{Department of Mathematics and Statistics,  University of Helsinki, Finland and Institute for Logic, Language and Computation, University of Amsterdam, The Netherlands}

\date{}
\maketitle

\begin{abstract}

A logical approach to Bell's Inequalities of quantum mechanics has been introduced by Abramsky and Hardy \cite{abramsky}.  We point out that the logical Bell's Inequalities of \cite{abramsky} are provable in the probability logic of Fagin, Halpern and Megiddo \cite{fagin}. Since it is now considered empirically established that quantum mechanics violates Bell's Inequalities, we introduce a modified probability logic, that we call quantum team logic, in which Bell's Inequalities are not provable, and prove a Completeness Theorem for this logic. For this end we generalise the team semantics of dependence logic \cite{MR2351449} first to probabilistic team semantics,  and then to what we call quantum team semantics.\end{abstract}


\section{Introduction}

Quantum logic was introduced by Birkhoff and von Neumann \cite{Birkhoff} to account for non-classical  phenomena in quantum physics.  Several other formulations have been suggested since.
Starting from   probabilistic propositional logic, which unsurprisingly turns  out to be inadequate for quantum physics, we   introduce here a new propositional logic,  called {\em quantum team logic}. The idea is to take advantage of some features of {\em team semantics} \cite{MR2351449} in order to model phenomena of quantum physics such as non-locality and entanglement. These phenomena were first emphasised by the famous paper of Einstein-Podolsky-Rosen  \cite{Einstein}, and then more conclusively by a result of J. S. Bell \cite{Bell}, known as Bell's Theorem.

In classical propositional logic the meaning of a sentence can be defined in terms of  truth-value assignments to the proposition symbols. In so-called {\em team semantics} of \cite{MR2351449}, the meaning of a sentence is defined in terms of sets of truth-value assignments, called {\em teams}. The advantage of this switch is that it becomes possible to define the meaning of a proposition symbol {\em depending on} or being {\em independent of} another proposition symbol. In this paper we do not discuss dependence or independence, but instead use team semantics to investigate the related concept of correlation of truth-values of proposition symbols in teams of assignments. In particular we use team semantics to define two different propositional logics. The first is a  propositional logic adequate for reasoning about expected truth values of propositional formulas. We show that Bell's Inequalities are provable in this logic. We then introduce another similar logic, quantum team logic, and show that the kind of Bell's Inequalities that can be violated, is not provable. The situation is a manifestation in logical terms of the recognised fact that  assigning probabilities to observables is not enough to explain correlations of entangled particles.   Both of our  logics extend and are based on \cite{fagin}.  Our approach is inspired by \cite{abramsky}.

An essential feature of quantum phenomena is that they are probabilistic. It is therefore natural in any attempt to model quantum physics by propositional logic to allow probabilistic truth-values. We accomplish this by considering {\em multi-teams}, that is, teams in which truth-value assignments occur with certain probabilities. This is our first step, and we call the resulting logic {\em probabilistic team logic}.

Multi-teams can be seen as results of experiments (the fact that the elements of the team take values from $\{0,1\}$ is not essential). E.g. someone throws a bowling ball at a rack of four pins. The result can be described by a function $f:\{0,1,2,3\}\to\{0,1\}$ ($f(i) = 1$ if pin $i$ is knocked down). When the experiment is repeated several times the results form a team. From this team one can calculate e.g. the probability of the event that either both pins $0$ and $1$ are knocked down or neither of them is knocked down. In our propositional logic this is the same as the (expected) truth value of the propositional formula $p_0 \leftrightarrow p_1$ in the multi-team.

The physical observations violating Bell's Inequalities, as well as quantum theoretic computations to the same effect, show that correlations between observations concerning entangled particles are stronger than can be explained by probabilities of individual (even hidden) variables. This leads us to define the more general concept of {\em quantum team}. Every multi-team is a quantum team but not conversely. Experiments demonstrating the violation of Bell's Inequalities give practical examples of quantum teams. 
Not all quantum teams correspond to quantum mechanical experiments because even so-called maximal violations of Bell's Inequalities can be manifested by quantum teams.

By giving the meaning of propositional symbols in terms of quantum teams we define {\em quantum team logic} and show that cases of the violation of Bell's Inequalities are simply examples of sentences of quantum team logic that are not valid. We give a proof system for our quantum team logic, based on \cite{fagin}, and prove a Completeness Theorem. We propose that our quantum team logic formalises probabilistic reasoning in quantum physics in perfect harmony with the non-locality phenomenon revealed by Bell's Inequalities. 
However, since quantum teams cover more than quantum mechanics, our quantum team logic does not formalize {\em exactly} the reasoning in quantum physics. Finding the logic for exact reasoning is left as an open problem.
	
	


\section{Notation}

We use $\omega$ to denote the set of natural numbers, $\omega^*$ to denote $\omega - \left\{ 0 \right\}$, and  $\mathcal{P}_{\omega}({\omega})$ to denote the set of non-empty finite subsets of $\omega$. 
We use $p_0, p_1, \ldots$ to denote proposition symbols. 
For a proposition symbol $p_i$ and $d\in 2$ we use $p_i^d$ to denote $p_i$, if $d=1$, and 
$\neg p_i$, if $d=0$.
We use the notation $( a_i )_{i < n}$
 for a sequence of $n$ elements $a_i$.
 
\section{Multi-teams}\label{prob_team}

A good source of teams for our purpose is the following Alice-Bob experiment:

\begin{itemize}
\item Alice has two registers $A_1$ and $A_2$ which both can contain a binary digit. 
\item Bob has two registers $B_1$ and $B_2$ which both can contain a binary digit.
\item The experiment consists of Alice and Bob both choosing one of their registers and reading the content, resulting in a tuple $(x_1,y_1,x_2,y_2)$, where $x_1\in\{A_1,A_2\}$, $y_1\in\{0,1\}$, $x_2\in\{B_1,B_2\}$ and $y_2\in\{0,1\}$. 

\end{itemize}
  
Each result $(x_1,y_1,x_2,y_2)$ of the Alice-Bob experiment can be thought of as an assignment of truth-values to proposition symbols $p_0,\ldots,p_3$ with the intention:

$$
\begin{array}{lcclcl}
p_0& \mbox{ is true} &\mbox{ iff }& \mbox{ Alice chose $A_1$}&\mbox{ i.e.}& \mbox{ $x_1=A_1$}\\
p_1& \mbox{ is true} &\mbox{ iff }& \mbox{ Alice read $1$}&\mbox{ i.e.}& \mbox{  $y_1=1$}\\
p_2& \mbox{ is true} &\mbox{ iff }& \mbox{ Bob chose $B_1$}&\mbox{ i.e.}& \mbox{  $x_2=B_1$}\\
p_3& \mbox{ is true} &\mbox{ iff }& \mbox{ Bob read $1$}&\mbox{ i.e.}& \mbox{  $y_2=1$}\\
\end{array}
$$	

A possible set of assignments arising in this way is in Table~\ref{team1}. Note that the table has  repeated rows, so we cannot identify the table with the {\em set} of the assignments constituting the table without losing some information. On the other hand, teams are {\em sets} of assignments. Thus  Figure~\ref{team1} does not represent a team in the sense of \cite{MR2351449}, but rather a {\em multi-team}.

\begin{figure}[h]
$$\begin{array}{|c|c|c|c|c|}
\hline
\phantom{a} & p_0 & p_1 & p_2 & p_3 \\
\hline
	0 & 1 & 1 & 0 & 1 \\
	1 & 1 & 1 & 1 & 1 \\
	2 & 1 & 1 & 1 & 1 \\
	3 & 1 & 1 & 1 & 0 \\
	4 & 0 & 0 & 1 & 1 \\
	5 & 0 & 0 & 0 & 0 \\
	6 & 0 & 0 & 0 & 0 \\
	7 & 0 & 0 & 0 & 0 \\	
\hline
\end{array}
$$\caption{Example of a multi-team  \label{team1}}\end{figure}

\def\dom{\mbox{dom}}
\begin{definition}[Multi-team]
A {\em {multi-team}} is a pair $X=(\Omega,\tau)$, where $\Omega$ is a non-empty set and $\tau$ is a function such that
   $\dom(\tau)=\Omega$ and if $i\in \Omega$, then  $\tau(i)$ is an assignment for one and the same non-empty set of proposition symbols, denoted by $\dom(X)$. The {\em size} of the multi-team is the cardinality $|\Omega|$ of $\Omega$.
\end{definition}
 Note, that an {ordinary team} $X$, i.e. a set of assignments, can be thought of as the {multi-team} $(\Omega,\tau)$, where $X=\Omega$ and $\tau(i)=i$ for all $i\in X$. 

A finite multi-team of size $n$ gives rise to the concept of a probability of an individual assignment:
$$P_X(v)=\frac{|\{i\in \Omega : \tau(i)=v\}|}{n}.$$ This extends canonically to a definition of  the probability (or expected value) of  a propositional formula  $\phi$:
		\[ [\phi]_X = P(\left\{ i \in \Omega \, | \, \tau(i) (\phi) = 1 \right\}) .\]
In fact, a finite multi-team is just a finite ordinary team $X$ endowed with a probability distribution on $X$. For infinite multi-teams the situation is a little different and calls for a new definition:

	\begin{definition}[Probability team] 
A {\em {probability team}} is a tuple $(\Omega,\F,P,\tau)$, where $\Omega$ is a set, $\F$ is a $\sigma$-algebra on $\Omega$, $P$ is a probability measure on $(\Omega,\F)$ and $\tau$ is a measurable function such that
   $\dom(\tau)=\Omega$ and if $i\in \Omega$, then  $\tau(i)$ is an assignment for one and the same set of proposition symbols, denoted by $\dom(X)$. 
\end{definition}

In this paper the main focus is on finite teams.



Suppose now $X = (\Omega,\tau)$ is a finite multi-team of size $n$ and $U\subseteq \dom(X)$. We can define a new multi-team $(\Omega,\tau_U)$ by letting $\tau_U(i)=\tau(i)\restriction U$. For each assignment $v$ on $U$ we define 
$$P_{X,U}(v)=\frac{|\{i\in \Omega : \tau_U(i)=v\}|}{n}.$$We write $P_{U}$ when $X$ is clear from the context. We can now make a table of the values $P_U(v)$ for various $U$ and $v$. 
For the multi-team of Figure~\ref{team1} and for $U=\{p_0,p_1\}$ we get  Table~\ref{prob_table0}. We have denoted the four possible assignments for $\{p_0,p_1\}$ as $(1,1),(0,1),(1,0)$ and $(0,0)$ with the obvious meaning. 
\begin{table}[h]
$$\begin{array}{|c|c|c|c|c|}
\hline
\phantom{a} & (1, 1) & (0,1) & (1,0) & (0,0) \\
\hline
	\{p_0, p_1\} & 1/2 & 0 & 0 & 1/2 \\
\hline
\end{array}
$$\caption{Example of a probability table \label{prob_table0}}\end{table}

It is relevant from the point of view of multi-teams arising is quantum physical experiments to consider a whole collection of subsets $U$ of $\dom(X)$ at the same time. We call a collection $\U=\{U_j : j\in J\}$ of subsets of $\dom(X)$ a {\em cover of $B$} if $\bigcup_{j\in J}U_j=B$. For two collections $\U$ and $\U'$ of sets we define $$\U\le\U'\iff\forall U\in\U\exists U'\in\U'(U\subseteq U').$$ 

\begin{definition}[Probability Table \cite{abramsky}\footnote{In \cite{abramsky} what we call probability tables are referred to as {\em probability models}.}]\label{def_prob_table}
Suppose $B$ is a finite set of proposition symbols and $\U$ a cover of $B$. A {\em probability table 	
for $B$ and $\U$} is a function $U\mapsto d_U$ on $\U$, where $d_U$ is a  probability distribution on the possible truth-value assignments $s$ for the proposition symbols in $U$ (i.e. $d_U(s)\in[0,1]$ and $\Sigma_{s}d_U(s)=1$).  	
\end{definition}

When each set $U$ in $\U$ has the same size $k$, the probability table is particularly easy to draw as a matrix as we can fix the truth-value assignments by reinterpreting $\{p_{i_0},\ldots,p_{i_{k-1}}\}$, where $i_0<\ldots<i_{k-1}$, as $\{p_0,\ldots,p_{k-1}\}$. With this convention all $U$ have the same truth-value assignments $v_0,\ldots,v_l$. See Figure~\ref{protab5}.

\begin{figure}
\begin{center}
\begin{tabular}{|c|c|c|c|}
\hline
    & $v_0$ &\ldots & $v_l$\\ 
\hline
$U_1$& $a^1_1$&\ldots&$a^1_l$\\
\vdots&\vdots&\vdots&\vdots\\
$U_s$& $a^s_1$&\ldots&$a^s_l$\\
\hline
\end{tabular} 
\end{center}
\caption{Probability table\label{protab5}} 
\end{figure}

\begin{definition}
If $X=(\Omega,\tau)$ is a finite multi-team  of size $n$ and $\U$ is a cover of $B\subseteq\dom(X)$, then the {\em associated probability table for $B$ and $\U$} is the function $U\mapsto d_U$ on $\U$, where $d_U$ is the  probability distribution $$d_U(v)= P_{X,U}(v)$$ on the possible truth-value assignments  for the proposition symbols in $U$.  	 
\end{definition}

	In Table~\ref{prob_table1} we have an example of a probability table for
	$\{p_0,p_1,p_2,p_3\}$ and  $\mathcal{U} = \left\{ \left\{ 0, 1 \right\}, \left\{ 0, 3 \right\}, \left\{ 1, 2 \right\}, \left\{ 2, 3 \right\} \right\}$, associated with the multi-team of Figure~\ref{team1}.

\begin{table}[h]
$$\begin{array}{|c|c|c|c|c|}
\hline
\phantom{a} & (1, 1) & (0,1) & (1,0) & (0,0) \\
\hline
	(p_0, p_1) & 1/2 & 0 & 0 & 1/2 \\
	(p_0, p_3) & 3/8 & 1/8 & 1/8 & 3/8 \\
	(p_1, p_2) & 3/8 & 1/8 & 1/8 & 3/8 \\
	(p_2, p_3) & 3/8 & 1/8 & 1/8 & 3/8 \\
\hline
\end{array}
$$\caption{Probability table \label{prob_table1} associated with Figure~\ref{team1}}\end{table}

	
		


%

\section{Logical Bell Inequalities}

John Stewart Bell showed in 1964 that spins of a pair of entangled particles
manifest correlations which cannot be explained by associating probabilities to spins of the individual particles in different directions, even if so-called ``local" hidden variables are allowed. Bell used the mathematical model of quantum mechanics for his result but the correlations in question have subsequently been verified by experiments. Bell's result is usually interpreted as a strong non-locality of the physical world. On the other hand, this non-locality has given rise to  quantum cryptography and more generally to quantum information theory. 

Abramsky and Hardy \cite{abramsky} presents a very logical formulation of Bell's result and we follow his presentation in this overview section. We present some details for completeness and refer the reader to \cite{abramsky} for further details.

The probability table we use for deriving Bell's Theorem is in  in Table~\ref{prob_table2}.  
	\begin{table}[h]
		$$\begin{array}{|c|c|c|c|c|}
		\hline
		\phantom{a} & (1, 1) & (0,1) & (1,0) & (0,0) \\
		\hline
			(p_0, p_1) & 1/2 & 0 & 0 & 1/2 \\
			(p_0, p_3) & 3/8 & 1/8 & 1/8 & 3/8 \\
			(p_1, p_2) & 3/8 & 1/8 & 1/8 & 3/8 \\
			(p_2, p_3) & 1/8 & 3/8 & 3/8 & 1/8 \\
		\hline
		\end{array}
	$$\caption{Bell's table  \label{prob_table2}}\end{table}
	Consider the Alice-Bob experiment mentioned in the introduction. Let us enrich the framework by imagining that a pair of (entangled) particles are sent to Alice and Bob. Let us decide that what we called Alice's register $A_1$ is actually a measurement of the spin of the particle that Alice has in direction $0^{\circ}$,   Alice's register $A_2$ is  a measurement of the spin of the particle that Alice has in direction $60^{\circ}$,  Bob's register $B_1$ is actually a measurement of the spin of the particle that Bob has in direction $180^{\circ}$, and finally Bob's register $B_2$ is  a measurement of the spin of the particle that Bob has in direction $120^{\circ}$. 
	
	Let us denote\footnote{This is different choice than before in Section~\ref{prob_team}.}
			$$\begin{array}{lcl}
			p_0 &=& \text{``Alice measurement at $0^{\circ}$ has outcome $\uparrow$."},\\
			p_1 &=& \text{``Bob measurement at $180^{\circ}$ has outcome $\uparrow$."},\\
			p_2 &=& \text{``Alice measurement at $60^{\circ}$ has outcome $\uparrow$."},\\
			p_3 &=& \text{``Bob measurement at $120^{\circ}$ has outcome $\uparrow$."},
	\end{array}$$
%
	Both quantum physical computations and actual experiments show that  Table~\ref{prob_table2} is the resulting probability table. However,   Table~\ref{prob_table2} is not the probability table associated with any multi-team. 
%
%
%
%
%
We give the proof, as presented in \cite{abramsky}, for completeness.
The method of  \cite{abramsky} is based on observations about expected values of propositional formulas. 

For this end, suppose $X=(\Omega,\tau)$ a multi-team the domain of which contains the proposition symbols of some given propositional formulas
$(\phi_j)_{j < k}$. Then
	\[ \begin{array}{rcl}
	1 - [\bigwedge_{j < k} \phi_j]_X & = & [\bigvee_{j < k}\neg\phi_j]_X \\
	 					& = & P (\mbox{\Large $\{$ } \!\! i \in X \, | \, \tau(i)(\bigvee_{j < k}\neg\phi_j) = 1 \mbox{\Large $\}$}) \\
	 					& = & P (\bigcup_{j < k} \left\{ i \in X \, | \, \tau(i)(\neg\phi_j) = 1 \right\}) \\
						& \leq & \sum_{j < k} P(\left\{ i \in X \, | \, \tau(i)(\neg \phi_j) = 1 \right\}) \\
	 					& = & \sum_{j < k} [\neg \phi_j]_X \\
	 					& = & \sum_{j < k} (1 - [\phi_j]_X) \\
	 					& = & k - \sum_{j < k} [\phi_j]_X. \\
\end{array}	\]
Hence
	
\begin{equation}\label{star}
 \sum_{j < k} [\phi_j]_X \leq k-1 + [\bigwedge_{j < k} \phi_j]_X. 
  \end{equation}
Furthermore if the formula $\bigwedge_{j < k} \phi_j$ is contradictory (in the sense of propositional logic), then $[\bigwedge_{j < k} \phi_j]_X =0$. Thus, the inequality (\ref{star}) becomes
	\begin{equation}\label{starstar} \sum_{j < k} [\phi_j]_X \leq k-1. \end{equation}
	Inequalities of this form (\ref{starstar}) are of great importance in foundations of quantum mechanics. In \cite{abramsky} they are called {\em logical Bell's inequalities}. 
	
	Suppose now that the probability table represented in Table~\ref{prob_table2} arises from a multi-team. That is, there is a 
	multi-team $X=(\Omega,\tau)$ with $\{p_0,p_1,p_2,p_3\}\subseteq \dom(X)$ such that 
	Table~\ref{prob_table2} 
	is the associated probability table for $\{p_0,p_1,p_2,p_3\}$ and $\U$.
	Consider now the following propositional formulas:
	$$\begin{array}{lcl}
		 \phi_0 &=& (p_0 \wedge p_1) \vee (\neg p_0 \wedge \neg p_1)\\
		\phi_1 &=& (p_0 \wedge p_3) \vee (\neg p_0 \wedge \neg p_3)\\
		\phi_2 &=& (p_1 \wedge p_2) \vee (\neg p_1 \wedge \neg p_2)\\
		\phi_3 &=& (\neg p_2 \wedge p_3) \vee (p_2 \wedge \neg p_3)\\
\end{array}$$

	Looking at Table~\ref{prob_table2} it is easy to notice that $[\phi_0]_X = 1$ and $[\phi_j]_X = \frac{6}{8}$ for $j = 1, 2, 3$. Furthermore,  the formula $\bigwedge_{j < 4} \phi_j$ is clearly contradictory. But then by (\ref{starstar}) we must have that
	\[ \sum_{j < 4}[\phi_j]_X = 1 + 3\cdot\frac{6}{8} = 3 + \frac{1}{4} \leq 3, \]
a contradiction. 

	Thus, Table~\ref{prob_table2} can not arise from a multi-team, because it violates the inequality (\ref{starstar}) by $\frac{1}{4}$. 
One consequence of this, when combined with existing actual measurements, is the remarkable result that
 the polarization of a photon cannot be independently measured
 in two different directions simultaneously.

	It is possible to construct probability tables consistent with quantum mechanics that violate (\ref{starstar}) by $1$, and so achieve {\em maximal} violation of the inequality (remember that the probability of a formula can not be greater than $1$). Tables~\ref{popescu_box} and \ref{GHZ_state} are emblematic examples of this. In \cite{abramsky} and \cite{7961} a general theory of probability tables (and generalizations thereof) is developed. A notion of {\em global section} is introduced and a strict hierarchy of classes of tables is defined: non-local tables, contextual tables and strongly contextual tables. As shown there, the first class corresponds exactly to the family of tables which violate a logical Bell's Inequality, while the third to the family of tables which maximally violate a logical Bell's Inequality. 
	
	\begin{table}[h]
$$\begin{array}{|c|c|c|c|c|}
\hline
\phantom{a} & (1, 1) & (0,1) & (1,0) & (0,0) \\
\hline
	(p_0, p_1) & 1/2 & 0 & 0 & 1/2 \\
	(p_0, p_3) & 1/2 & 0 & 0 & 1/2 \\
	(p_1, p_2) & 1/2 & 0 & 0 & 1/2 \\
	(p_2, p_3) & 0 & 1/2 & 1/2 & 0 \\
\hline
\end{array}
$$\caption{Popescu-Rohrlich box (cfr. \cite{abramsky}) \label{popescu_box}}\end{table}

\begin{table}[h]
$$\begin{array}{|c|c|c|c|c|c|c|c|c|}
\hline
\phantom{a} & 		111 & 110 & 101 & 100 & 011 & 010 & 001 & 000 \\
\hline
	(p_0, p_1, p_4) & 0 & 1/4 & 1/4 & 0 & 1/4 & 0 & 0 & 1/4 \\
	(p_0, p_1, p_5) & 1/4 & 0 & 0 & 1/4 & 0 & 1/4 & 1/4 & 0 \\
	(p_0, p_1, p_5) & 1/4 & 0 & 0 & 1/4 & 0 & 1/4 & 1/4 & 0 \\
	(p_0, p_1, p_4) & 1/4 & 0 & 0 & 1/4 & 0 & 1/4 & 1/4 & 0 \\
\hline
\end{array}
$$\caption{$\mathrm{GHZ}$ state (cfr. \cite{abramsky}) \label{GHZ_state}}\end{table}

	
\section{Quantum Teams}\label{sec_quantum_teams}

A quantum team is a multi-team in which some values are indeterminate, reflecting the situation in quantum phenomena that some variables cannot be measured together. In the quantum theoretic Alice-Bob experiment the truth-values of propositions $p_0$ and $p_2$ (also $p_1$ and $p_3$) cannot be both determined. We isolate this phenomenon by specifying a sequence $\Q=(Q_i)_{i< m}$ of finite sets of proposition symbols. Intuitively, each $Q_i$ is a set of elementary propositions that {\em can} be measured together.

	\begin{definition}[Quantum team]\label{quantum_team} 
Suppose $\Omega$ is a finite set.	Let  $\Q=(Q_i)_{i\in\Omega}$ be a sequence of finite non-empty sets of proposition symbols. A {\em quantum team} on $\Q$ is a pair $X=(\Omega,\tau)$ such that $\tau(i)$ is a truth-value assignment to the proposition symbols in $Q_i$ for each $i\in \Omega$. We call $\{Q_i:i\in\Omega\}$ the {\em support} of $X$ and denote it $\sp(X)$. The set $\bigcup_{i\in\Omega}Q_i$ is called the {\em domain} of $X$ and denoted $\dom(X)$. 
	
\end{definition}

Note that a multi-team is always a quantum team as we can let $Q_i=\dom(X)$ for all $i\in \Omega$. On the other hand, obviously a quantum team need not be a multi-team. 

If $X=(\Omega,\tau)$ is a quantum team and $j\in \dom(X)\setminus Q_i$, then $\tau(i)(j)$ is not determined and we call it {\em indeterminate}. Indetermined values arise in quantum physics naturally. For example, a particle has a spin in every direction, but once it is measured in one direction, spin in other directions cannot be measured independently. In graphical representations of teams we represent indeterminate values using the symbol $-$.

%
	To make clear this convention we give an example of quantum team. Let 
	\[ Q_i = \begin{cases} \left\{ 0, 1 \right\} &\mbox{if } i < 8 \\ 
							\left\{ 0, 3 \right\} &\mbox{if } 8 \leq i < 16 \\ 
							\left\{ 1, 2 \right\} &\mbox{if } 16 \leq i < 24 \\ 
							\left\{ 2, 3 \right\} &\mbox{if } 24 \leq i < 32, 
\end{cases} \]
then Figure~\ref{team2} depicts a $(Q_i)_{i< 32}$-quantum team.
%

{\small
\begin{figure}[h]
$$\begin{array}{|c|c|c|c|c|}
\hline
\phantom{a} & p_0 & p_1 & p_2 & p_3 \\
\hline
	0 & 1 & 1 & - & - \\
	1 & 1 & 1 & - & - \\
	2 & 1 & 1 & - & - \\
	3 & 1 & 1 & - & - \\
	4 & 0 & 0 & - & - \\
	5 & 0 & 0 & - & - \\
	6 & 0 & 0 & - & - \\
	7 & 0 & 0 & - & - \\
	8 & 1 & - & - & 1 \\	
	9 & 1 & - & - & 1 \\	
   10 & 1 & - & - & 1 \\	
   11 & 0 & - & - & 1 \\	
   12 & 1 & - & - & 0 \\	
   13 & 0 & - & - & 0 \\	
   14 & 0 & - & - & 0 \\	
   15 & 0 & - & - & 0 \\	
   16 & - & 1 & 1 & - \\	
   17 & - & 1 & 1 & - \\	
   18 & - & 1 & 1 & - \\	
   19 & - & 0 & 1 & - \\	
   20 & - & 1 & 0 & - \\	
   21 & - & 0 & 0 & - \\	
   22 & - & 0 & 0 & - \\	
   23 & - & 0 & 0 & - \\	
   24 & - & - & 1 & 1 \\	
   25 & - & - & 1 & 0 \\	
   26 & - & - & 1 & 0 \\	
   27 & - & - & 1 & 0 \\	
   28 & - & - & 0 & 1 \\	
   29 & - & - & 0 & 1 \\	
   30 & - & - & 0 & 1 \\
   31 & - & - & 0 & 0 \\				
\hline
\end{array}
$$\caption{Example of a quantum team \label{team2}}\end{figure}
}

Given a finite set $U$ of proposition symbols and a quantum team $(\Omega,\tau)$ on $(Q_i)_{i\in\Omega}$, we let $\Omega_{U} = \left\{ i \in\Omega \, | \, U \subseteq Q_i \right\}$. We use this notation only if $\Omega_U\ne\emptyset$. 

Suppose  $X=(\Omega,\tau)$ is a quantum team on $(Q_i)_{i\in\Omega}$ and 
$\{U\}\le\sp(X)$. We can define a new quantum team $X_U=(\Omega_U,\tau_U)$ by letting $\tau_U(i)=\tau(i)\restriction U$ for $i\in \Omega_U$. For each assignment $v$ on $U$ we define 
$$P_{X,U}(v)=\frac{|\{i\in \Omega_U : \tau_U(i)=v\}|}{|\Omega_U|}.$$ We write $P_{U}$ when $X$ is clear from the context. 
This extends canonically to a definition of the probability of a propositional formula $\phi$ with its proposition symbols in $U$ such that $\Omega_{U} \neq \emptyset$: 
		\[ [\phi]_{X,U} = P_{X,{U}}(\left\{ i \in \Omega_{U} \, | \, \tau_U(i) (\phi) = 1 \right\}) .\]
If $U = \mathrm{Var}(\phi)$ we simply write $[\phi]_{X}$, instead of $[\phi]_{X,\mathrm{Var}(\phi)}$. 

\begin{definition}
Suppose we have a quantum team $X=(\Omega,\tau)$ on $\Q$, a set $B\subseteq\dom(X)$ and a cover $\U$ of $B$ such that $\U\le\Q$. The {\em associated probability table  for $B$ and $\U$} is the following
 function $U\mapsto d_U$ on $\U$:
$$d_U(v)= P_{X,U}(v).$$ 
 \end{definition}

	A moment's reflection shows that Bell's table (i.e. Table~\ref{prob_table2}) is the probability table associated with  the team represented in Figure~\ref{team2}, and $A$ and $\mathcal{U}$ as in the description of Table~\ref{prob_table2}. Similarly, it is possible to see that the Popescu-Rohrlich box (i.e. Table~\ref{popescu_box}) and the $\mathrm{GHZ}$ state (i.e. Table~\ref{GHZ_state}) arise from quantum teams\footnote{In the case of the Popescu-Rohrlich box just modify Table~\ref{prob_table2} changing the first non-indeterminate entry in lines 11, 12, 19, 20, 24 and 31.}.
\medskip

\begin{lemma}
 Every probability table with rational probabilities is the associated table of some quantum team.
\end{lemma}

\begin{proof} Suppose $\mathcal{U} = ( U_i )_{i < n}$ is a cover of a set $B$ of proposition symbols. Let $(d_{U_i})_{i <n}$ be a probability table for $ \mathcal{U}$ and $B$ with rational values.  
Let $s^k_t$, $0\le t<2^{|U_i|}$ list all truth assignments for proposition symbols in $U_i$.
For $i < n$, let $a^i_t\in\omega$ and $b_i\in\omega^*$ be such that $d_{{U}_i}(s^i_t) = {a^i_t}/{b_i}$ for $t<2^{|U_i|}$. Let $m = \sum_{i < n} b_i$. 
For $\sum_{i < k} b_i \leq j < \sum_{i < k+1} b_i$, let $Q_j = U_k$. Let $X=(\Omega,\tau)$ be the  quantum team on $(Q_j)_{j < m}$ such that $X=m$ and for $\sum_{i < k} b_i \leq j < \sum_{i < k+1} b_i$ and $\sum_{p < t} a^k_p \leq j - \sum_{i < k} b_i < \sum_{p < t + 1} a^k_p$ we have that $\tau(j) = s^k_t$. Then $X$ is as desired.
\end{proof}

\section{Probabilistic team logic}


As observed in \cite{abramsky}, Bell's Inequalities can be expressed in terms of expected values of simple propositional formulas. We introduce now a version of propositional logic in which Bell's Inequalities can be expressed and proved. Our approach is based on \cite{fagin}.  Since experiments, as well as theoretical computations, violate Bell's Inequalities, our probabilistic team logic is not appropriate for quantum physics. In the next section we present a new logic, quantum team logic, in which the ``false" Bell's Inequalities are not provable, and which therefore has a better chance to model adequately the logic of quantum phenomena. 

	Following \cite{fagin}, we formulate a logic that is capable of expressing rational inequalities. The syntax and deductive system of this logic are the same as those of \cite{fagin}. The semantics is different, but equiexpressive with the original one, as we shall see. We call this logic {\em probabilistic team logic} $(\mathrm{PTL})$. Paradigm examples of formulas of $\ptl$ are formulas that we write as 
	$$\phi_0+\ldots + \phi_{k-1} \leq k-1,$$ expressing
	the logical Bell's Inequality 
	$$[\phi_0]_X+\ldots + [\phi_{k-1}]_X \leq k-1.$$ 

	\begin{definition} Suppose $\phi_0,\ldots,\phi_k$ are propositional formulas, $(a_j)_{j \le k} \in \mathbb{Z}^k$ and $c \in \mathbb{Z}$, then $$a_0 \phi_0+\ldots +a_{k} \phi_{k} \geq c$$ is an atomic formula of $\mathrm{PTL}$. 		
\end{definition}
		
	\begin{definition}\label{def_form} The set of formulas of $\mathrm{PTL}$ is defined as follows:
		\begin{itemize}
			\item Atomic formulas are formulas;
			\item If $\alpha$ is a formula, then $\neg\alpha$ is a formula;
			\item If $\alpha$ and $\beta$ are formulas, then $\alpha \wedge \beta$ is a formula.
	\end{itemize}		
\end{definition}

	Disjunction and implication are defined in terms of negation and conjunction in the usual manner. We shall use obvious abbreviations, such as $\phi -\psi \geq c$ for $\phi + (-1)\psi \geq c$, $\phi \geq \psi$ for $\phi - \psi \geq 0$, $\phi \leq c$ for $-\phi \geq -c$, $\phi < c$ for $\neg (\phi \geq c)$, $\phi = c$ for $(\phi \geq c) \wedge (\phi \leq c)$ and $\phi = \psi$ for $(\phi \geq \psi) \wedge (\phi \leq \psi)$. A formula such as $\phi \geq \frac{1}{3}$ can be viewed as an abbreviation for $3\phi \geq 1$; we can always allow rational numbers in our formulas as abbreviations for the formula that would be obtained by clearing the denominators.
		
	\begin{definition}[Semantics] Suppose $X=(\Omega,\tau)$ is a multi-team and $\alpha$ a formula of $\mathrm{PTL}$ with  propositional symbols in $\dom(X)$. We define by induction on $\alpha$ the relation $X\models \alpha$ in the following way:
		\begin{itemize}
			\item $X \models a_0 \phi_0+\ldots + a_{k-1} \phi_{k-1} \geq c$ iff $a_0 [\phi_0]_X+\ldots +a_{k-1} [\phi_{k-1}]_X \geq c$;
			\item $X \models \neg\alpha$ iff $X \not\models \alpha$;
			\item $X \models \alpha \wedge \beta$ iff $X \models \alpha$ and $X \models \beta$.
	\end{itemize}		
\end{definition}

	We say that $\alpha$ is {\em satisfiable} if there is a multi-team $X$ such that $X \models \alpha$, and that $\alpha$ is valid, in symbols $\models \alpha$, if $X \models \alpha$ for every multi-team $X$. Notice that the arguments presented in Section~\ref{prob_team} show that for any sequence of propositional formulas $(\phi_0, ..., \phi_{k-1})$ such that $\bigwedge_{j < k} \phi_j$ is unsatisfiable we have that the formula
	\[ \sum_{j < k}\phi_j \leq k-1\]
is a validity of $\mathrm{PTL}$.
In particular, for $\phi_0, \phi_1, \phi_2$ and $\phi_3$ as in Section~\ref{prob_team} we have that the formula
\begin{equation}\label{fs}
 \phi_0+\phi_1+\phi_2+\phi_3 \leq 3
  \end{equation}
is a validity of $\mathrm{PTL}$.

	\begin{definition}[Deductive system]\label{ded_system1} The deductive system of $\mathrm{PTL}$ breaks into the  following three sets of rules.
		\[ \text{Propositional reasoning}\]
		\begin{enumerate}[A)]
			\item All instances of propositional tautologies.
			\item If $\alpha \rightarrow \beta$ and $\alpha$, then $\beta$ (modus ponens).
	\end{enumerate}	
	   \[ \text{Probabilistic reasoning}\]
		\begin{enumerate}[A)]\setcounter{enumi}{2}
			\item $\phi \geq 0$.
			\item $\phi \vee \neg \phi = 1$.
			\item $\phi \wedge \psi + \phi \wedge \neg \psi = \phi$ (additivity).
			\item If $\phi \equiv \psi$ in propositional logic, then $\phi = \psi$.		
	\end{enumerate}
		\[ \text{Linear inequalities}\]
		\begin{enumerate}[A)]\setcounter{enumi}{6}
			\item $\phi \geq \phi$.
			\item $\sum_{j < k} a_j \phi_j \geq c$ $\Leftrightarrow$ $\sum_{j < k} a_j \phi_j + 0 \psi \geq c$.
			\item $\sum_{j < k} a_j \phi_j \geq c$ $\Leftrightarrow$ $\sum_{j < k} a_{\sigma(j)} \phi_{\sigma(j)} \geq c$ (for $\sigma$ permutation on $k$).
			\item $\sum_{j < k} a_j \phi_j \geq c \wedge \sum_{j < k} b_j \phi_j \geq d$ $\Leftrightarrow$ $\sum_{j < k} (a_j + b_j) \phi_j \geq c + d$.
			\item $\sum_{j < k} a_j \phi_j \geq c$ $\Leftrightarrow$ $\sum_{j < k} da_j \phi_j \geq c$ (for $d >0$).
			\item $\sum_{j < k} a_j \phi_j \geq c \vee \sum_{j < k} a_j \phi_j \leq c$.
			\item $\sum_{j < k} a_j \phi_j \geq c$ $\Rightarrow$ $\sum_{j < k} a_j \phi_j > d$ (for $c > d$).
	\end{enumerate}	
\end{definition}

	A deduction is a sequence of formulas $(\alpha_0 , ... , \alpha_{n-1})$ such that each $\alpha_i$ is either an instance of the axioms of our deductive system or follows from one or more formulas of $\left\{ \alpha_0, ... , \alpha_{i-1} \right\}$ by one of its rules. We say that $\alpha$ is provable, in symbols $\vdash \alpha$, if there is a deduction $(\alpha_0 , ... , \alpha_{n-1})$ with $\alpha = \alpha_{n-1}$. We say that $\alpha$ is consistent if $\nvdash \alpha \rightarrow \bot$ and inconsistent otherwise.

	\begin{theorem}[Completeness]\label{compl_PTL} Let $\alpha$ be a formula of $\mathrm{PTL}$. Then
		\[ \vdash \alpha \;\; \Leftrightarrow \;\; \models \alpha. \]	
\end{theorem}

	\begin{proof} It follows from Theorem~2.2 and Theorem~2.4 of \cite{fagin}, by noticing that the {\em small model} of Theorem~2.4 can be taken to be uniform by adding points to the sample space.
	
\end{proof}

In consequence, even the ``false" Bell's Inequalities, such as (\ref{fs}) above that correspond to phenomena that can be violated by quantum mechanical computations as well as by actual experiments, are provable in $\ptl$. Thus $\ptl$ is not the ``right" logic for arguing about probabilities in quantum physics. In the next section a better candidate is introduced. 

\section{Quantum team logic}

	 In this section we generalize $\mathrm{PTL}$ to a more expressive logic: quantum team logic ($\mathrm{QTL}$). The syntax of this logic is more complicated than that of $\mathrm{PTL}$. This modification is necessary in order to account for the fine structure of quantum teams and prove a completeness theorem. Instead of atomic formulas of the form
	 $$a_0 \phi_0+\ldots +a_{k-1} \phi_{k-1}\geq c,$$ as in $\ptl$, we adopt atomic formulas of the more complicated form
	 $$a_0 (\phi_0; V_0)+\ldots +a_{k-1} (\phi_{k-1}; V_{k-1}) \geq c$$ in order to capture the phenomenon, prevalent in quantum physics, that there are limitations as to what observables can be measured simultaneously.

	\begin{definition} Suppose $(\phi_j)_{j\le k}$ are propositional formulas, $(a_j)_{j < k} \in \mathbb{Z}^k$, $c \in \mathbb{Z}$ and $(V_j)_{j < k}$ a sequence of finite sets of proposition symbols, so that the proposition symbols of $\phi_j$ are  in $V_j$ for every $j < k$. Then 
	$$a_0 (\phi_0; V_0)+\ldots +a_{k-1} (\phi_{k-1}; V_{k-1}) \geq c$$ is an atomic formula of $\mathrm{QTL}$. 	
\end{definition}

	\begin{definition}\label{def_quantum_form} The set of formulas of $\mathrm{QTL}$ is defined as follows:
		\begin{itemize}
			\item atomic formulas are formulas;
			\item if $\alpha$ is a formula, then $\neg\alpha$ is a formula;
			\item if $\alpha$ and $\beta$ are formulas, then $\alpha \wedge \beta$ is a formula.
	\end{itemize}		
\end{definition}

	Also in this case we shall use some abbreviations, such as $(\phi; V) - (\psi; V) \geq c$ for $(\phi; V) + (-1)(\psi;V) \geq c$, $(\phi \geq \psi; V)$ for $(\phi; V) - (\psi; V) \geq 0$, $(\phi; V) \leq c$ for $-(\phi; V) \geq -c$, $(\phi; V) < c$ for $\neg ((\phi; V) \geq c)$, $(\phi; V) = c$ for $((\phi; V) \geq c) \wedge ((\phi; V) \leq c)$ and $(\phi = \psi; V)$ for $((\phi \geq \psi; V)) \wedge ((\phi \leq \psi; V))$. Furthermore, if in a formula $\sum_{j < k} a_j (\phi_j; V_j) \geq c$ we have that $V_0 = \cdots = V_{k-1}$, we simply write $(\sum_{j < k} a_j \phi_j; V_0) \geq c$. As for $\mathrm{PTL}$, we will allow rational numbers in our formulas as abbreviations for the formula that would be obtained by clearing the denominators.

	\begin{definition}[Elementary components] Let $\alpha$ be a formula of $\mathrm{QTL}$, we define the elementary components of $\alpha$, in symbols $\mathrm{EC}(\alpha)$, by induction on $\alpha$ in the following way:
		\begin{enumerate}[i)]
			\item $\mathrm{EC}(\sum_{j < k} a_j (\phi_j; V_j) \geq c) = \left\{ (\phi_j; V_j) \, | \, j < k \right\}$;
			\item $\mathrm{EC}(\neg\alpha) = \mathrm{EC}(\alpha)$;
			\item $\mathrm{EC}(\alpha \wedge \beta) = \mathrm{EC}(\alpha) \cup \mathrm{EC}(\beta)$.		
	\end{enumerate}		
\end{definition}

	Given an elementary component $(\phi, V)$, we call $V$ the support of $(\phi, V)$. It makes sense to define this notion for any formula of $\mathrm{QTL}$.

	\begin{definition}[Support] Let $\alpha$ be a formula of $\mathrm{QTL}$, we define the support of $\alpha$, in symbols $\mathrm{Sp}(\alpha)$, by induction on $\alpha$ in the following way:
		\begin{enumerate}[i)]
			\item $\mathrm{Sp}(\sum_{j < k} a_j (\phi_j; V_j) \geq c) = \left\{ V_j \, | \, j < k \right\}$;
			\item $\mathrm{Sp}(\neg\alpha) = \mathrm{Sp}(\alpha)$;
			\item $\mathrm{Sp}(\alpha \wedge \beta) = \mathrm{Sp}(\alpha) \cup \mathrm{Sp}(\beta)$.
	\end{enumerate}		
\end{definition}
	
	We isolate two important classes of formulas of $\mathrm{QTL}$.
		
	\begin{definition} Let $\alpha$ be a formula of $\mathrm{QTL}$ and $((\phi_j; V_j))_{j < k}$ an enumerataion of its elementary components. 
		\begin{enumerate}[i)]
			\item We say that $\alpha$ is {\em classical} if $V_0 = \cdots = V_{k-1}$.
			\item We say that $\alpha$ is {\em normal} if $V_j = \mathrm{Var}(\phi_j)$ for every $j < k$.
 	\end{enumerate}
\end{definition}

	Normal formulas will be denoted omitting supports. Thus, syntactically (not semantically) they look exactly like the formulas of $\mathrm{PTL}$, and will be denoted using the same conventions used there. Classical formulas convey the same semantic content as formulas of $\mathrm{PTL}$, from this their name.
	
	 Given a quantum team $X=(\Omega,\tau)$, we let
		\[ \mathrm{Sp}(X) = \left\{ \mathrm{dom}_X(i) \, | \, i \in\Omega \right\},\]
where with $\mathrm{dom}_X(i)$ we mean the domain of the function $\tau(i)$. We call $\mathrm{Sp}(X)$ the support of $X$. Notice that for any $U \in \mathcal{P}_{\omega}(\omega)$ such that there is $V \in \mathrm{Sp}(X)$ with $U \subseteq V$ we have that $\Omega_V = \left\{ i \in \Omega \, | \, U \subseteq \mathrm{dom}_X(i) \right\} \neq \emptyset$.
	
	\begin{definition}[Semantics] Let $\alpha$ be a formula of $\mathrm{QTL}$ and $X$ a quantum team with $\mathrm{Sp}(\alpha) \leq \mathrm{Sp}(X)$. We define by induction on $\alpha$ the relation $X \models \alpha$ in the following way:
		\begin{itemize}
			\item $X \models \sum_{j < k} a_j (\phi_j; V_j) \geq c$ iff $\sum_{j < k} a_j [\phi_j]_{X,{V_j}} \geq c$;
			\item $X \models \neg\alpha$ iff $X \not\models \alpha$;
			\item $X \models \alpha \wedge \beta$ iff $X \models \alpha$ and $X \models \beta$.
	\end{itemize}		
\end{definition}

	We say that $\alpha$ is satisfiable if there is a quantum team $X$ with $\mathrm{Sp}(\alpha) \leq \mathrm{Sp}(X)$ such that $X \models \alpha$, and that $\alpha$ is valid, in symbols $\models \alpha$, if $X \models \alpha$ for every quantum team $X$ with $\mathrm{Sp}(\alpha) \leq \mathrm{Sp}(X)$. 
	As evident, with respect to normal formulas the only difference between $\mathrm{PTL}$ and $\mathrm{QTL}$ is that the set of teams with respect to which we define the semantics for $\mathrm{QTL}$ is wider than that used for $\mathrm{PTL}$ (remember that multi-teams are particular cases of quantum teams). This allows for the modeling of non-classical phenomena. Notice indeed that in the case of $\mathrm{QTL}$, for $\phi_0, \phi_1, \phi_2$ and $\phi_3$ as in Section~\ref{prob_team} we have that the formula
		
\begin{equation}\label{tag1}
\sum_{j < 4}\phi_j \leq 3  
\end{equation}
is {\em not} a validity of $\mathrm{QTL}$, because the formula
		\[ \sum_{j < 4}\phi_j \geq 3 + \frac{1}{4} \]
is satisfied by the team represented in Figure~\ref{team2}. As a matter of facts, an even stronger negation of (\ref{tag1}) is satisfiable, namely
		\begin{equation}\label{tag2} \sum_{j < 4}\phi_j = 4, \end{equation}
because the team from which the Popescu-Rohrlich box arises satisfies (\ref{tag2}). 
The fact that these formulas are consistent should be no mystery, as indeed  fundamental  laws of probability  presuppose the kinds of classical structures that one
does not find in  $\mathrm{QTL}$,  as the remark below shows.
	
	\begin{remark}\label{failure_add} Let $X$ be the quantum team represented in Table~\ref{team3} and $\alpha$ the following formula:
		\[ p_0 \wedge p_1 + p_0 \wedge \neg p_1 = p_0.\]
Then $X \not\models \alpha$ because 
		\[ [p_0 \wedge p_1]_X + [p_0 \wedge \neg p_1]_X = \frac{1}{2} + \frac{1}{2} \neq \frac{1}{2} = [p_0]_X. \]
\begin{table}[h]
		$$\begin{array}{|c|c|c|c|}
		\hline
	\phantom{a} & p_0 & p_1 & p_3\\
	\hline
				0 & 1 & 1 & - \\
				1 & 1 & 0 & - \\
				2 & 0 & - & 1 \\
				3 & 0 & - & 0 \\	
	\hline
\end{array}
$$\caption{Counterexample to additivity \label{team3}}\end{table}		
\end{remark}

	Remark~\ref{failure_add} also shows that the deductive system described in Definition~\ref{ded_system1} is {\em not} sound with respect to the quantum team semantics given in the present section, because  the additivity axiom (rule E)) is not respected. The following remark shows that also rule F) is violated.
	
		\begin{remark}\label{failure_ruleE} Let $X$ be the quantum team represented in Table~\ref{team4} and let
			\[ \phi = (p_0 \vee \neg p_0) \wedge p_1 \text{ and } \psi = (p_2 \vee \neg p_2) \wedge p_1.\]
	Then clearly $\phi \equiv \psi$ (in propositional logic) but $X \not\models \phi = \psi$ because 
			\[ [\phi]_X = \frac{1}{2} \neq 0 = [\psi]_X. \]
	\begin{table}[h]
		$$\begin{array}{|c|c|c|c|}
		\hline
	\phantom{a} & p_0 & p_1 & p_2 \\
	\hline
			    0 & 1 & 1 & - \\
			    1 & 1 & 0 & - \\
			    2 & - & 0 & 0 \\
			    3 & - & 0 & 0 \\	
	\hline
\end{array}
$$\caption{Counterexample to rule F) \label{team4}}\end{table}		
\end{remark}

	Notice that the teams represented in Remarks~\ref{failure_add} and~\ref{failure_ruleE} are compatible with the thought experiment described in the introduction. As indeed, if we think of $p_0$ and $p_2$ to be the outcome of Alice's measurements, and $p_1$ and $p_3$ to be the outcome of Bob's measurements (for some choice of angles), then the presence of indeterminates\footnote{Remember the definition of indeterminate that we gave after Definition~\ref{quantum_team}. Indeterminates are just entries of the matrix representing the team that are not defined in some rows but that are defined in some others.} in the teams is compatible with the predictions of quantum mechanics (i.e. we can not measure the spins of the same particle at two different angles).
	
	We now come to the deductive system of $\mathrm{QTL}$. At first sight, this system may look a little technical, but it expresses exactly what happens on the semantic side of $\mathrm{QTL}$ (and in fact we will show that it is complete). The system should be thought as a family of localizations of the deductive system of $\mathrm{PTL}$.
	
	\begin{definition}[Deductive system]\label{ded_system2} The deductive system of $\mathrm{QTL}$ is parametrized by finite subsets of $\mathcal{P}_{\omega}(\omega)$. For any $\mathcal{V} \subseteq_{\mathrm{fin}} \mathcal{P}_{\omega}(\omega)$ it breaks into the  following four sets of $\mathcal{V}$-rules, where each one of the formulas involved is such that $\mathrm{Sp}(\alpha) \leq \mathcal{V}$.
		\[ \text{Propositional reasoning}\]
		\begin{enumerate}[A)]
			\item $\vdash_{\mathcal{V}} \alpha$, for $\alpha$ a propositional tautology.
			\item If $\vdash_{\mathcal{V}} \alpha \rightarrow \beta$ and $\vdash_{\mathcal{V}} \alpha$, then $\vdash_{\mathcal{V}} \beta$ (modus ponens).
\end{enumerate}	
	   \[ \text{Probabilistic reasoning}\]
		\begin{enumerate}[A)]\setcounter{enumi}{2}
			\item $\vdash_{\mathcal{V}}(\phi; V) \geq 0$.
			\item $\vdash_{\mathcal{V}}(\phi \vee \neg \phi; V) = 1$.
			\item $\vdash_{\mathcal{V}}(\phi \wedge \psi + \phi \wedge \neg \psi = \phi; V)$ (additivity).
			\item If $\phi \equiv \psi$ in propositional logic, then $\vdash_{\mathcal{V}} (\phi = \psi; V)$.		
	\end{enumerate}
		\[ \text{Linear inequalities}\]
		\begin{enumerate}[A)]\setcounter{enumi}{6}
			\item $\vdash_{\mathcal{V}} (\phi \geq \phi; V)$.
			\item $\vdash_{\mathcal{V}} \sum_{j < k} a_j (\phi_j; V_j) \geq c$ $\Leftrightarrow$ $\vdash_{\mathcal{V}} \sum_{j < k} a_j (\phi_j; V_j) + 0 (\psi; V) \geq c$.
			\item $\vdash_{\mathcal{V}} \sum_{j < k} a_j (\phi_j; V_j) \geq c$ $\Leftrightarrow$ $\vdash_{\mathcal{V}} \sum_{j < k} a_{\sigma(j)} (\phi_{\sigma(j)}; V_{\sigma(j)}) \geq c$ (for $\sigma$ permutation on $k$).
			\item $\vdash_{\mathcal{V}} \sum_{j < k} a_j (\phi_j; V_j) \geq c \wedge \sum_{j < k} b_j (\phi_j; V_j) \geq d$ $\Leftrightarrow$ $\vdash_{\mathcal{V}} \sum_{j < k} (a_j + b_j) (\phi_j; V_j) \geq c + d$.
			\item $\vdash_{\mathcal{V}} \sum_{j < k} a_j (\phi_j; V_j) \geq c$ $\Leftrightarrow$ $\vdash_{\mathcal{V}} \sum_{j < k} da_j (\phi_j; V_j) \geq c$ (for $d >0$).
			\item $\vdash_{\mathcal{V}} \sum_{j < k} a_j (\phi_j; V_j) \geq c \vee \sum_{j < k} a_j (\phi_j; V_j) \leq c$.
			\item $\vdash_{\mathcal{V}} \sum_{j < k} a_j (\phi_j; V_j) \geq c$ $\Rightarrow$ $\vdash_{\mathcal{V}} \sum_{j < k} a_j (\phi_j; V_j) > d$ (for $c > d$).
	\end{enumerate}
	    \[ \text{Change of support}\]
		\begin{enumerate}[A)]\setcounter{enumi}{13}	
			\item $\vdash_{\mathcal{V}} (\phi, V) = 0$ and $V \subseteq V' \in \mathcal{V}$
$\Rightarrow$ $(\phi, V') = 0$.
			\item $\vdash_{\mathcal{V}} (\phi, V) = 1$ and $V \subseteq V' \in \mathcal{V}$
$\Rightarrow$ $(\phi, V') = 1$.
	\end{enumerate}
\end{definition}	

	Let $\mathcal{V} \subseteq_{\mathrm{fin}} \mathcal{P}_{\omega}(\omega)$. A $\mathcal{V}$-deduction is a sequence of formulas $(\alpha_0 , ... , \alpha_{n-1})$ such that $\mathrm{Sp}(\alpha_i) \leq \mathcal{V}$ for every $i < n$ and $\alpha_i$ is either an instance of $\mathcal{V}$-axioms of our deductive system or follows from one or more formulas of $\left\{ \alpha_0, ... , \alpha_{i-1} \right\}$ by one of its $\mathcal{V}$-rules. We say that $\alpha$ is provable, in symbols $\vdash \alpha$, if there is an $\mathrm{Sp}(\alpha)$-deduction $(\alpha_0 , ... , \alpha_{n-1})$ with $\alpha = \alpha_{n-1}$. We say that $\alpha$ is consistent if $\nvdash \alpha \rightarrow \bot$ and inconsistent otherwise.
	
	Before analyzing the problem of completeness of $\mathrm{QTL}$ we notice that axioms G) - M) axiomatize the set of valid inequality formulas. We make this point clear. Based on \cite{fagin}, we define a logical system for linear inequalities, which we call $\mathrm{LinIneq}$. Let $\mathrm{IndVar} = \left\{ v_i \, | \, i \in \omega \right\}$ be a countable set, called the set of individual variables.
	
	\begin{definition} Let $k \in \omega^*$, $(a_j)_{j < k} \in \mathbb{Z}^k$, $c \in \mathbb{Z}$ and $(x_j)_{j<k} \in \mathrm{IndVar}^k$. Then $\sum_{j < k} a_j x_j \geq c$ is an atomic formula of $\mathrm{LinIneq}$. The formulas of $\mathrm{LinIneq}$ are boolean combinations of atomic formulas of $\mathrm{LinIneq}$.
		
\end{definition}

	\begin{definition} Let $f(\vec{x})$ be a formula of $\mathrm{LinIneq}$ with variables from $\vec{x} = (x_0, ..., x_{n-1})$ and $A: \vec{x} \rightarrow \mathbb{R}$. We define by induction on $f$ the relation $A \models f$ in the following way:
		\begin{enumerate}[i)]
			\item $A \models \sum_{j < k} a_j x_j \geq c$ iff $\sum_{j < k} a_j A(x_j) \geq c$;
			\item $A \models \neg f$ iff $X \not\models f$;
			\item $A \models f \wedge g$ iff $X \models f$ and $X \models g$.
	\end{enumerate}		
\end{definition}
	
	\begin{definition}\label{ded_system3} The deductive system of $\mathrm{LinIneq}$ breaks into the two following sets of rules.
	\[ \text{Propositional reasoning}\]
	\begin{enumerate}[a)]
		\item All instances of propositional tautologies.			
		\item If $f \rightarrow g$ and $f$, then $g$ (modus ponens).
		\end{enumerate}	
	\[ \text{Linear inequalities}\]
			\begin{enumerate}[a)]\setcounter{enumi}{2}
				\item $x \geq x$.
				\item $\sum_{j < k} a_j x_j \geq c$ $\Leftrightarrow$ $\sum_{j < k} a_j x_j + 0y \geq c$.
				\item $\sum_{j < k} a_j x_j \geq c$ $\Leftrightarrow$ $\sum_{j < k} a_{\sigma(j)} (x_{\sigma(j)} \geq c$ (for $\sigma$ permutation on $k$).
				\item $\sum_{j < k} a_j x_j \geq c \wedge \sum_{j < k} b_j x_j \geq d$ $\Leftrightarrow$ $\sum_{j < k} (a_j + b_j) x_j \geq c + d$.
				\item $\sum_{j < k} a_j x_j \geq c$ $\Leftrightarrow$ $\sum_{j < k} da_j x_j \geq c$ (for $d >0$).
				\item $\sum_{j < k} a_j x_j \geq c \vee \sum_{j < k} a_j x_j \leq c$.
				\item $\sum_{j < k} a_j x_j \geq c$ $\Rightarrow$ $\sum_{j < k} a_j x_j > d$ (for $c > d$).
	\end{enumerate}				
\end{definition}	

	\begin{lemma} Let $f$ be a formula of $\mathrm{LinIneq}$, then 
		\[ \vdash f \;\; \Leftrightarrow \;\; \models f. \]
\end{lemma}

	\begin{proof} See~\cite[Theorem~4.3]{fagin}.
		
\end{proof}

	Given a formula $f(\vec{x})$ of $\mathrm{LinIneq}$, we say that $f$ has a {\em rational solution} if there is $A: \vec{x} \rightarrow \mathbb{R}$ such that $\mathrm{ran}(A) \subseteq \mathbb{Q}$ (i.e. the set of rational numbers).
	
		\begin{lemma}\label{axiomatization_lin_ineq} Let $f$ be a formula of $\mathrm{LinIneq}$. If $f$ is consistent (in $\mathrm{LinIneq}$), then $f$ has a rational solution.
	\end{lemma}

	\begin{proof} See~\cite[Theorem~4.9]{fagin}.
\end{proof}

	We now come back to the problem of completeness of $\mathrm{QTL}$.
	
	\begin{theorem}[Completeness]\label{compl_QTL} Let $\alpha$ be a formula of $\mathrm{QTL}$. Then
		\[ \vdash \alpha \;\; \Leftrightarrow \;\; \models \alpha. \]	
\end{theorem}

	\begin{proof} Soundness is easy. Regarding completeness, we show that every consistent formula is satisfiable. Let then $\alpha$ be a consistent formula and $\mathrm{Sp}(\alpha) = \mathcal{V}$. Given $V \in \mathcal{P}_{\omega}({\omega})$ and $s \in 2^V$, we let $\phi_s  = \bigwedge_{v \in V} p_{v}^{s(v)}$. Let
		\[ \beta_{\mathcal{V}}^0 = \bigwedge_{\substack{ V, V' \in \, \mathcal{V} \\ V \subseteq V'}} (\bigwedge_{s \in 2^V}( (\phi_s = 0; V) \rightarrow (\phi_s = 0; V'))),\]
and
		\[ \beta_{\mathcal{V}}^1 = \bigwedge_{\substack{ V, V' \in \, \mathcal{V} \\ V \subseteq V'}} (\bigwedge_{s \in 2^V}( (\phi_s = 1; V) \rightarrow (\phi_s = 1; V'))).\]
		Define $\beta_{\mathcal{V}} = \beta_{\mathcal{V}}^0 \wedge \beta_{\mathcal{V}}^1$. Notice that because of rules N) and O) we have that $\beta_{\mathcal{V}}$ is provable. Let also
				\[ \gamma_{\alpha}^0 = (\bigwedge_{V \in \mathcal{V}} (\sum_{s \in 2^{V}} (\phi_{s}; V) = 1)) \wedge (\bigwedge_{V \in \mathcal{V}} (\bigwedge_{s \in 2^V} ((\phi_{s}; V) \geq 0))), \]
		and
				\[ \gamma_{\alpha}^1 = \bigwedge_{(\phi; V) \in \mathrm{EC}(\alpha)} ((\phi; V) = \sum_{\substack{s \in 2^{\mathrm{Var(\phi)}} \\ \!\!s \models \phi}}(\phi_s; V)).\]
		Define $\gamma_{\alpha} = \gamma_{\alpha}^0 \wedge \gamma_{\alpha}^1$. Notice that also $\gamma_{\alpha}$ is provable, this is because of rule C) and the following lemma.
		
	\begin{lemma} Let $\phi$ be a propositional formulas and $V \in \mathcal{P}_{\omega}(\omega)$ with $\mathrm{Var}(\phi) \subseteq V$. Then the formula
		\begin{equation}\label{starstarstar}
		(\phi; V) = \sum_{\substack{s \in 2^{\mathrm{Var(\phi)}} \\ \!\!s \models \phi}}(\phi_s; V) \end{equation}
is provable in $\mathrm{QTL}$\footnote{Notice that the provability of the first conjunct in $\gamma_{\alpha}^0$ follows from this by taking $\phi = (\bigwedge_{v \in V} p_v) \vee \neg (\bigwedge_{v \in V} p_v)$. In fact from axiom D) and (\ref{starstarstar}) we have that
\[ 1 = ((\bigwedge_{v \in V} p_v) \vee \neg (\bigwedge_{v \in V} p_v)); V) = \sum_{\substack{s \in 2^{\mathrm{Var(\phi)}} \\ \!\!s \models \phi}}(\phi_s; V) = \sum_{s \in 2^{V}} (\phi_{s}; V). \]
}.
\end{lemma}

	\begin{claimproof} It follows from the fact that without the support the formula is provable in $\mathrm{PTL}$, for details see \cite[Lemma~2.3.]{fagin}.
\end{claimproof}

	Let now $\delta = \alpha \wedge \beta_{\mathcal{V}} \wedge \gamma_{\alpha}$. This formula is consistent, because $\alpha$ is consistent by hypothesis, $\beta_{\mathcal{V}}$ and $\gamma_{\alpha}$ are provable, and $\mathrm{Sp}(\beta_{\mathcal{V}}), \mathrm{Sp}(\gamma_{\alpha}) \subseteq \mathcal{V}$. Let $\left\{ \delta_i \, | \, i < l\right\}$ be the set of atoms occurring in $\delta$. Thinking of $\delta$ as a propositional formula in the propositional variables $(\delta_i)_{i < l}$, it is clear that the formula
	\[ \delta \leftrightarrow \bigvee_{\substack{S \in 2^l \\ \!\!S \models \delta}} (\bigwedge_{i < l} \delta_i^{S(i)})\]
is provable in $\mathrm{QTL}$, because $\mathrm{QTL}$ has all the validities of propositional logic in its deductive system. Thus, from the consistency of $\delta$ we can infer the existence of an assignment $S: \left\{ \delta_i \, | \, i < l\right\} \rightarrow 2$ such that $\bigwedge_{i < l}\delta_i^{S(i)}$ is consistent. Let $\delta^{*} = \bigwedge_{i < l}\delta_i^{S(i)}$, for $S$ such an assignment. We show that there is a quantum team $X$ such that $X \models \delta^{*}$. This suffices to establish the satisfiability of $\delta$ in $\mathrm{QTL}$ and thus of $\alpha$. 

	Let $\left\{ x_i \, | \, i < m \right\}$ be the set of elementary components of $\delta^{*}$. Thinking of $\delta^{*}$ as system of linear inequalities in the individual variables $(x_i)_{i < m}$, we have that $\delta^{*}$ is a formula of $\mathrm{LinIneq}$. Because of axioms G) - M), given that $\delta^{*}$ is consistent in $\mathrm{QTL}$ we must have that $\delta^{*}$ is consistent in $\mathrm{LinIneq}$. Thus, by Lemma~\ref{axiomatization_lin_ineq}, we can infer that $\delta^{*}$ has a rational solution. 
	
	Let $(q_i)_{i < e}$ be a rational solution of $\delta^{*}$ (thought as a system of linear inequalities). For any $V \in \mathcal{V}$ we build a multi-team $X(V)$ with domain $V$ following the information encoded in $(q_i)_{i < e}$. Let $V \in \mathcal{V}$, $(s_i)_{i < h}$ an enumeration of the truth assignments to proposition symbols in $V$, wehnever the rational number corresponding to the component $(\phi_s, V)$ is different from $0$, and $(q_{k_i})_{i < h}$ an enumeration of the rational numbers corresponding to $(s_i)_{i < h}$. Notice that because of $\gamma_{\alpha}^0$ the sequence $(q_{k_i})_{i < h}$ can not be empty, all the elements of the sequence are positive and $\sum_{i < h} q_{k_i} = 1$. Let $t \in \omega^*$ be such that $\frac{a_i}{t} = q_{k_i}$ for every $i < h$. We define $X(V)=(\Omega,\tau)$, where $\Omega=t$ and $\tau: t \rightarrow 2^V$ is defined by
	\[ \tau(z) = \begin{cases}  s_0 &\mbox{if } z < a_0 \\ 
							    s_1 &\mbox{if } a_0 \leq z < (a_0 + a_1) - 1 \\ 
							    ... &\mbox{if } ... \\ 
							    s_{h-1} &\mbox{if } \sum_{i < h-1} a_i \leq z < t. 
\end{cases} \]
Notice that for every $(\phi, V) \in \mathrm{EC}(\alpha)$ we have that 
	\[ [\phi]_{X(V),V} = [\phi]_{X(V)} = q, \]
where $q$ is the rational number corresponding to the elementary component $(\phi; V)$. This is because of $\gamma_{\alpha}^1$ and the fact that every $X(V)$ is a multi-team.

	We now linearly order $\mathcal{V}$ satisfying the requirement that if $V \subsetneq V'$ then $V' < V$. Let $(V_0, ..., V_{d-1})$ be the enumeration of $\mathcal{V}$ that follows this order. By induction on $d$, we define quantum teams $(X^i)_{i < d}$, $X^i=(\Omega^i,\tau^i)$, such that for every $j \leq i < d$ and $(\phi, V_{j}) \in \mathrm{EC}(\delta^*)$ we have that
	\begin{equation}\label{tag1s}
	 (\Omega^i)_{V_{j}} = \left\{ z \in \Omega^i \, | \, V_{j} \subseteq \mathrm{dom}_{X^{i}}(z) \right\} \neq \emptyset, \end{equation}
	\begin{equation}\label{tag2s} [\phi]_{X^i,{V_{j}}} = [\phi]_{X(V_{j})}.\end{equation}
Clearly $X^{d-1}$ will be such that $X^{d-1} \models \delta^{*}$.
\newline {\bf Base case).} $X_0 = X(V_0)$. Notice that requirements (\ref{tag1s}) and (\ref{tag2s}) are trivially satisfied.
\newline {\bf Inductive case).} Suppose we have defined $X^{i}$. We are going to define $X^{i+1}$ by gluing $X(V_{i+1})$ to $X^{i}$ without altering probabilities. Let $p$ and $m$ be the number of lines in $X(V_{i+1})$ and $X^{i}$, respectively. There are two cases.
\medskip

\noindent{\textit{Case 1).}} There is no $V \in \mathcal{V}$ such that $V_{i+1} \subsetneq V$. Let $X^{i+1}$ be the team obtained extending $X^i$ with $p$ many lines with domain $V_{i+1}$ and assigning functions in $2^{V_{i+1}}$ according to the values appearing in the rows of $X(V_{i+1})$. By the fact that we extend $X^i$ (and in particular we remove none of the rows of $X^i$), and the fact that we add a strictly positive number of rows with domain $V_{i+1}$ we have that (\ref{tag1s}) is satisfied. Furthermore, for every $s \in 2^V$ we have that
	\[ [\phi_s]_{X^{i+1},{V_{i+1}}} = [\phi_s]_{X(V_{i+1})},\]
and the probabilities of the other supports remain unaltered, and so (\ref{tag2s}) is also satisfied.%
\medskip

\noindent{\textit{Case 2).}} There is at least one $V \in \mathcal{V}$ such that $V_{i+1} \subsetneq V$. Let now $m_{V_{i+1}} = \left\{ j < m \, | \, V_{i+1} \subseteq \mathrm{dom}_{X^{i}}(j) \right\}$ and $|m_{V_{i+1}}| = k$, i.e. the number of lines in $X^{i}$ where the support ${V_{i+1}}$ is defined. Notice that $k > 0$, because for any $V \in \mathcal{V}$ such that $V_{i+1} \subseteq V$ we have that $k \geq |m_V|$, and $m_V \neq \emptyset$ by inductive hypothesis. We extend the quantum team $X^{i}$ with $k(p-1)$ lines with domain $V_{i+1}$\footnote{Of course it is possible that $p = 1$, and so $k(p-1) = 0$. A moment reflection shows that this is not a problem.}. For $s \in 2^{V_{i+1}}$, let 
	\[ [\phi_s]_{X^{i},{V_{i+1}}} = \frac{b_s}{k} \;\;  \text{ and } \;\; [\phi_s]_{X(V_{i+1})} = \frac{a_s}{p}. \]
For every $s \in 2^V$, we let the assignment $s$ appear in $a_sk-b_s$ many of the new lines. 
\medskip

\begin{claim} We claim that this works, namely:
	\begin{enumerate}[i)]
		\item $0 \leq a_sk-b_s \leq k(p-1)$;
		\item $\sum_{s \in 2^V} (a_sk-b_s) = k(p-1)$.
	\end{enumerate}
\end{claim}
\begin{claimproof} Item ii) is easy, because
	\[ \begin{array}{rcl}
	\sum_{s \in 2^V} (a_sk-b_s) & = & \sum_{s \in 2^V} a_sk - \sum_{s \in 2^V} b_s\\
								& = & k\sum_{s \in 2^V} a_s - \sum_{s \in 2^V} b_s\\
								& = & kp - k\\
								& = & k(p-1). 
\end{array}	\]
We verify i). Let $s \in 2^V$, we distinguish three cases.
\medskip

\noindent{\em Case A)} $a_s = 0$. If $a_s = 0$, then $[\phi_s]_{X(V_{i+1})} = 0$, and so the two conjuncts expressing the formula $(\phi_s; V_{i+1}) = 0$ occur in $\delta^{*}$. But then, because of $\beta_{\mathcal{V}}^0$, for every $V \in \mathcal{V}$ such that $V_{i+1} \subsetneq V$ also the two conjuncts expressing the formula $(\phi_s; V) = 0$ occurs in $\delta^{*}$, and so $[\phi_s]_{X(V)} = 0$. Thus, by induction hypothesis, for any such $V$ we have that $[\phi_s]_{X^{i},{V}} = [\phi_s]_{X(V)} = 0$. Hence, $[\phi_s]_{X^{i},{V_{i+1}}} = 0$ because 
	\[m_{V_{i+1}} = \left\{ j < m \, | \, V_{i+1} \subseteq \mathrm{dom}_{X^{i}}(j) \right\} = \bigcup_{\substack{V \in \mathcal{V} \\ V_{i+1} \subsetneq V}} (\left\{ j < m \, | \, V \subseteq \mathrm{dom}_{X^{i}}(j) \right\}),\]
and so 
	\[ \left\{ j \in m_{V_{i+1}} \, | \, X^{i}(j)(\phi_s) = 1 \right\} = \bigcup_{\substack{V \in \mathcal{V} \\ V_{i+1} \subsetneq V}} (\left\{ j \in m_{V} \, | \, X^{i}(j)(\phi_s) = 1 \right\}) = \emptyset.\]
From which it follows that $a_sk-b_s = 0$.
\medskip

\noindent {\em Case B)} $a_s = p$. If $a_s = p$, then $[\phi_s]_{X(V_{i+1})} = 1$, and so the two conjuncts expressing the formula $(\phi_s; V_{i+1}) = 1$ occur in $\delta^{*}$. Thus, reasoning as in the case above we see that because of $\beta_{\mathcal{V}}^1$ we must have that $b_s = k$. From which it follows that $a_sk-b_s = k(p-1)$.
\medskip

\noindent {\em Case C)} $ 0< a_s < p$. Simply notice that
	\[ \begin{array}{rcl}
	 a_s < p & \Rightarrow & a_s + 1 \leq p \\
			 & \Rightarrow & ka_s + k \leq kp \\
		     & \Rightarrow & ka_s \leq k(p-1) \\
		     & \Rightarrow & ka_s - b_s \leq k(p-1),		
\end{array}	\]
and
	\[ \begin{array}{rcl}
	 0 < a_s & \Rightarrow & 1 \leq a_s \\
			 & \Rightarrow & k \leq ka_s \\
		     & \Rightarrow & 0 \leq k -b_s \leq ka_s - b_s. 
\end{array}	\] \end{claimproof}

Let $X^{i+1}$ be the quantum team resulting from the process described above. Requirement (\ref{tag1s}) is satisfied because also in this case we extend $X^i$, and already in $X^{i}$ there are $k > 0$ lines where the support $V_{i+1}$ is defined. Furthermore, for any $s \in 2^V$ we have that
	\[ [\phi_s]_{X^{i+1},{V_{i+1}}} = \frac{b_s + a_sk - b_s}{k + k(p-1)} = \frac{a_s}{p} = [\phi_s]_{X(V_{i+1})},\]
and the probabilities of the other supports remain unaltered, and so (\ref{tag2s}) is also satisfied. This concludes the proof of the theorem.

\end{proof}

Because the size of the team used in the above theorem can be computed from the formula $\alpha$, we get the following corollary:

\begin{corollary} The logic $\mathrm{QTL}$ is decidable.
\end{corollary}

\begin{remark}
In his famous {\em Lectures on Physics} \cite{MR0213079}  Richard Feynman explains to the students the ``Double-Slit Experiment":
Electrons are accelerated toward  a thin metal plate with two holes (hole 1 and hole 2) in it. Beyond the wall is another plate with a movable detector in front of it so that we can record an empirical probability distribution for the spot where the particle hits the second plate. It turns out that even when particles are sent one by one, the distribution has interference as if the particles went through both holes in a  wave-like fashion.
Feynman proposes the question whether the following proposition is true or false:
\smallskip

\begin{quote}
 Proposition A:	Each electron either goes through hole 1 or it goes through hole 2.
\end{quote} 
\smallskip

\noindent To test this proposition, Feynman puts in this hypothetical experiment  a light source near each hole so that we can observe an electron passing through that hole from the flash it gives as it scatters the light. It turns out that the probability distribution of  the spot where the particle hits the second plate has now no interference as if each particle indeed went
through hole 1 or through hole 2.  Feynman comments this as follows:
\smallskip

\begin{quote}
``Well," you say, ``what about Proposition A? Is it true, or is it not true, that the electron either goes through hole 1 or it goes through hole 2?" 
\end{quote}
\smallskip

\noindent We  use quantum team logic to model this riddle. Let us adopt the following notation for the basic propositions of this situation:
\begin{itemize}
\item $p_0$ = "we see the flash at hole 1".
\item  $p_1$ = "we see the flash at hole 2".
\item  $q_i $ = "the detector got the electron at distance $i$ from the center", for $i \in \mathbb{Z}$.
\end{itemize}
Let $\phi$ be the sentence $(p_0 \vee p_1) \wedge \neg (p_0 \wedge p_1)$. This sentence is Feynman's Proposition A. It is easy to construct a quantum team $X$ so that:

\begin{enumerate}
\item  $X \models \phi = 1$.
\item  $X \models (\phi\wedge q_i) \neq q_i.$
\end{enumerate}

\noindent Thus in this quantum team, even though the probability of $\phi$ is $1$, the probability of $\phi\wedge q_i$ is different from the probability of $q_i$. Although this defies intuition about probabilities, in quantum team logic it is just a feature, not a paradox.

\end{remark}

\begin{remark}
We have defined the semantics of our quantum team logic with reference to arbitrary quantum teams, whether they arise from actual quantum mechanical considerations or not. As the maximal violations of Bell's Inequality show, there are quantum teams that do not correspond to any actual experiments in quantum mechanics. This raises the following question:

\begin{quote}{\bf Open Question:} Can we axiomatize completely the formulas of quantum
team logic that are valid in quantum teams that correspond to quantum mechanical experiments? In particular, is that set of formulas recursive?
\end{quote} 
\end{remark}

\section{Conclusion}
	
	We introduced two new families of teams: multi-teams and quantum teams. The first family of teams models a notion of experiment which is compatible with classical mechanics but does not account for the predictions of quantum mechanics and experimental verifications thereof. The second family of teams is wider than the first and accounts for the non-locality phenomena which are typical of quantum mechanics. Based on these families of teams, we formulated two new logics: probabilistic team logic ($\mathrm{PTL}$) and quantum team logic $(\mathrm{QTL})$. $\mathrm{PTL}$ is only an adaptation of the system presented in \cite{fagin} to the framework of team semantics, while $\mathrm{QTL}$ is an original system, which we think appropriate for a logical analysis of the thought experiments considered in the foundations of quantum physics and the relative probability tables. The language of $\mathrm{QTL}$ is built up from rational inequalities, and the non-classical nature of  quantum teams allows for the satisfiability of rational inequalities expressing violations of Bell's Inequalities. Finally, we devised a deductive system for $\mathrm{QTL}$ and showed that this system is complete with respect to the intended semantics, making the logical treatment of the subject complete.

\end{document}